\documentclass[12pt,a4paper]{article}
\usepackage[utf8]{inputenc}
\usepackage{amsmath,amsthm,amsfonts,amsthm,amssymb,color}
\usepackage[left=1.5cm,top=1.5cm,right=1.5cm,bottom=2.5cm]{geometry}
\usepackage{authblk}
\usepackage{enumitem}
\usepackage{array,color,colortbl}
\usepackage{tikz}
  \usetikzlibrary{matrix,calc}
  \usetikzlibrary{decorations.markings,decorations.pathreplacing}
  \usetikzlibrary{patterns}
\usepackage{algorithm}
\usepackage[noend]{algpseudocode}
\usepackage{booktabs}
\usepackage{pdflscape}

\renewcommand{\=}{\equiv}

\newcommand{\n}{\tilde n}

\usepackage[noabbrev,capitalise]{cleveref}
\crefname{equation}{}{}
\Crefname{equation}{}{}

\def\sref#1{\S$\ref{#1}$}
\def\lref#1{Lemma~$\ref{#1}$}
\def\tref#1{Theorem~$\ref{#1}$}

\def\cjref#1{Conjecture~$\ref{#1}$}

\renewcommand{\geq}{\geqslant}
\renewcommand{\leq}{\leqslant}
\renewcommand{\ge}{\geqslant}
\renewcommand{\le}{\leqslant}

\usepackage{xspace}
\newcommand{\etal}{\emph{et al.}\xspace}
\newcommand{\ie}{\emph{i.e.}\xspace}

\tikzset{square matrix/.style={
    matrix of nodes,
    column sep=-\pgflinewidth, row sep=-\pgflinewidth,
    nodes={draw,
      minimum height=15pt,
      anchor=center,
      text width=13pt,
      align=center,
      inner sep=0pt
    },
  },
  square matrix/.default=2cm
}

\colorlet{trans}{blue!50}
\colorlet{trans2}{red!20}

\renewcommand\epsilon{\varepsilon}

\newcommand{\unkn}{\bullet}
\newcommand{\X}{\times\xspace}
\newcommand{\x}{$\X$\xspace}

\algnewcommand\algorithmicforeach{\textbf{for each}}
\algdef{S}[FOR]{ForEach}[1]{\algorithmicforeach\ #1\ \algorithmicdo}

\newcommand{\myComment}[1]{\Comment{{\color{blue} #1}}}

\theoremstyle{plain}
\newtheorem{theorem}{Theorem}
\newtheorem{lemma}[theorem]{Lemma}

\theoremstyle{definition}

\newtheorem{conjecture}[theorem]{Conjecture}
\newtheorem{example}[theorem]{Example}

\theoremstyle{remark}

\author{Darcy~Best\thanks{Research supported by Endeavour Postgraduate Scholarship and NSERC CGS-D.
}}
  \author{Kyle Pula
  }
\author{Ian~M.~Wanless\thanks{Corresponding author \texttt{ian.wanless@monash.edu}. Research supported by ARC grant DP150100506.}}
\affil{\small School of Mathematics, Monash University, Australia}
\title{Small Latin arrays have a near transversal}
\date{}

\begin{document}

\maketitle

\begin{abstract}
A {\em Latin array} is a matrix of symbols in which no symbol occurs
more than once within a row or within a column. A {\em diagonal} of an
$n\times n$ array is a selection of $n$ cells taken from different rows and
columns of the array. The {\em weight} of a diagonal is the number of
different symbols on it. We show via computation that every Latin array
of order $n\le11$ has a diagonal of weight at least $n-1$. A corollary
is the existence of near transversals in Latin squares of these orders.

More generally, for all $k\le20$ we compute a lower bound on the order
of any Latin array that does not have a diagonal of weight at
least $n-k$.

\bigskip

\end{abstract}
 
\section{\label{s:intro}Introduction}

In this article, we examine several different types of arrays of symbols.
An $m \times n$ array is \emph{row-Latin} (resp.~\emph{column-Latin}) if each row
(resp.~column) contains each symbol at most once.  If the array is both
row-Latin and column-Latin, then we say that the array is \emph{Latin}. An
\emph{equi-$n$-square} is an $n \times n$ array where each symbol is
represented exactly $n$ times, with no row or column restrictions.
A \emph{Latin square} is a Latin equi-$n$-square.

A \emph{diagonal} of an $n\times n$ array is a selection of $n$ cells
from different rows and columns of the array. The \emph{weight} of a
diagonal is the number of different symbols on it.  An \emph{entry} in
an array is a triple $(r,c,s)$ where $s$ is the symbol in cell $(r,c)$
of the array.  A \emph{partial transversal of length} $\ell$ is a set
of $\ell$ entries, each selected from different rows and columns of a
matrix, such that no two of the entries contain the same symbol.  In
an $m\times n$ matrix, partial transversals of length $\min(m,n)$ and
$\min(m,n)-1$, respectively, are known as \emph{transversals} and
\emph{near transversals}. Hence, in an $n\times n$ array, a diagonal
of weight $n$ is a transversal and a diagonal of weight $n-1$
\emph{contains} two near transversals.  Note that in our terminology,
partial transversals are defined for non-square arrays, but diagonals
are not.

It is known that there are a vast number of Latin squares that do not
contain a transversal \cite{CW17}. Thus, the best we can hope for in
general is a near transversal. The following conjecture has been
attributed to Brualdi (see \cite[p.103]{DKI}) and Stein \cite{Ste75}
and, in \cite{EHNS88}, to Ryser. It has recently been proved for
Cayley tables of finite groups \cite{GH20}.  For several
generalisations of the conjecture, in terms of hypergraphs, see \cite{AB09}
(although some of those generalisations have since been shown to fail
\cite{GRWW19}).

\begin{conjecture}\label{conj:brualdi}
  Every Latin square contains a near transversal.
\end{conjecture}

The main result of this paper is that every Latin array of order
$n\le11$ has a diagonal of weight at least $n-1$. In particular,
\cjref{conj:brualdi} holds for Latin squares of order $n\le11$.
Our method will be 
computational, but it is worth bearing in mind that the number
of Latin squares of orders up to 11 is too large to treat them
individually \cite{MW05}, and the number of Latin arrays of these
orders is presumably many orders of magnitude larger. Hence the
key to the viability of our computation is to eliminate candidate
counterexamples on the basis of limited partial information
about their structure. In particular, we will need to consider
\emph{partial arrays}, where some cells may be empty (contain no
symbol). Empty cells cannot be chosen in a partial transversal, but
may be included in a diagonal (in which case, they do not contribute
to the weight of that diagonal).

Since the result we are proving for small squares is more general than
\Cref{conj:brualdi}, we now review earlier attempts to broaden that
conjecture in various directions.  In 1975, Stein~\cite{Ste75} studied
transversals (under the name ``Latin transversals'') in
equi-$n$-squares. He made the following seven interrelated
conjectures:

\begin{conjecture}\label{conj:stein}
  \mbox{ }
  \begin{enumerate}
  \item Every equi-$n$-square has a near transversal.
  \item Every $n \times n$ array in which no symbol appears more than
    $n-1$ times has a transversal.
  \item Every $(n-1) \times n$ array in which no symbol appears more
    than $n$ times has a transversal.
  \item Every $(n-1) \times n$ row-Latin array has a transversal.
  \item Every $m \times n$ array (where $m < n$) in which no symbol
    appears more than $n$ times has a transversal.
  \item Every $(n-1) \times n$ array in which each symbol appears
    exactly $n$ times has a transversal.
  \item Every $m \times n$ array (where $m < n$) in which no symbol
    appears more than $m+1$ times has a transversal.
  \end{enumerate}
\end{conjecture}

Note that modern authors often include a requirement that for an array
to be row-Latin it should have the same symbols in each row. However,
Stein did not include that restriction, so in this paper we do not either.

Stein's conjectures are closely related to one another (with some 
being special cases of others). Unfortunately, all but one of these
conjectures has been disproven (only number (4) remains open). In
1998, Drisko~\cite{Drisko98} gave the (transpose of) the
following construction. The proof that we give is new, and in the spirit
of the Delta Lemma (see \cite{transurvey}). 
Here and henceforth, row and column indices always start at 0.

\begin{theorem}\label{t:Drisko}
  Let $m$ and $n$ be integers satisfying $m < n \leq 2m-2$. 
  Define an $m\times n$ column-Latin array $A=[a_{ij}]$ 
  on symbols $\{0,\dots,m-1\}$, by
\[
a_{ij}\=\begin{cases}
i&\bmod~{m}\text{ if $j\in\{0,\dots,m-2\}$},\\
i+1&\bmod~{m}\text{ if $j\in\{m-1,\dots,n-1\}$}.
\end{cases}
\]
Then $A$ has no transversals.
\end{theorem}

\begin{proof} 
Define $\Delta(i,j)=a_{ij}-i$ for $0\le i<m$ and $0\le j<n$.
Suppose that $A$ has a transversal on cells $(0,j_0),\dots,(m-1,j_{m-1})$.
Then since every symbol in $\{0,\dots,m-1\}$ appears in the transversal it
follows that
\[
\sum_{i=0}^{m-1}\Delta(i,j_i)\=\sum_{i=0}^{m-1}i-\sum_{i=0}^{m-1}i\=0\mod m.
\]
However, this congruence is impossible to satisfy, given that 
$\Delta(i,j)=0$ for $j\le m-2$ and $\Delta(i,j)\=1\mod m$ for 
$j>m-2$, so that
\[
\sum_{i=0}^{m-1}\Delta(i,j_i)\=|\{i:j_i>m-2\}|\not\=0\mod m.\qedhere
\]
\end{proof}

The $m = n-1$ case of \Cref{t:Drisko} is a direct counterexample to
\cref{conj:stein} parts (3), (5), (6) and (7). Also Stein and
Szab\'o~\cite{SS06} gave counterexamples to part (2) when $n\in\{5,6\}$.
Furthermore, Pokrovskiy and Sudakov~\cite{PS19} recently gave a
constructive proof of the following.

\begin{theorem}\label{th:stein-wrong}
For all sufficiently large $n$ there exists an equi-$n$-square that
does not have a partial transversal of length more than
$n-\frac{1}{42}\log n$.
\end{theorem}

This is a counterexample to \cref{conj:stein}(1). Moreover, 
Pokrovskiy and Sudakov showed how their result can be extended
to give further counterexamples to statements like
\cref{conj:stein}(2) as well. For some $n>e^{84}$, use
\cref{th:stein-wrong} to construct an equi-$n$-square $E$ of order $n$
with no partial transversal of length $n-2$. Now, form a matrix $A$ of
order $n+1$ by adding one row and one column to $E$, where the new row
and column contain $2n+1$ distinct symbols that do not appear in $E$.
Since at most two of these new symbols can be used in any partial
transversal, $A$ does not have a transversal.  By a similar argument,
we can pad an appropriate equi-$n$-square with either new symbols
or the original symbols to provide counterexamples to 
parts (3), (5), (6) and (7) that are of a different nature to the
counterexamples provided by \tref{t:Drisko}. Interestingly,
\cite{PS13} shows that almost all equi-$n$-squares have a transversal,
so these counterexamples may be viewed as atypical.

The only one of Stein's conjectures that remains unsolved is
(4). While we are quite unsure about this conjecture, it seems much
more promising if we also enforce the array to be column-Latin.

\begin{conjecture}\label{conj:brualdi-rectangle}
  Let $R$ be an $(n-1) \times n$ Latin array. Then $R$ contains a
  transversal.
\end{conjecture}

Note that this conjecture implies a strengthened form of
\cjref{conj:brualdi}, where we may choose which row (or column or
symbol) is not included in our partial transversal. It is an open
question (see~\cite{BW14}) whether an even stronger property holds
when $n$ is large: it may be that all large $(n-1) \times n$ Latin
arrays on $n$ symbols have a decomposition into transversals. Often
such arrays are called \emph{Latin rectangles}. Interestingly, it is
not possible to relax the requirement on the number of symbols. We
know there are many Latin squares without transversals \cite{CW17}. If
a column of previously unused symbols is appended to such a Latin
square we create a Latin array with no decomposition into
transversals. Alternatively, if we remove the first two rows and the
first column from the Cayley table of an elementary abelian group of
order $2^k>2$ then we get a $(2^k-2)\times(2^k-1)$ Latin array on
$2^k$ symbols, with no decomposition into transversals. This follows
from a result of Akbari and Alireza~\cite{AA04}. Nevertheless, 
Latin arrays that are in some sense far from being Latin squares
are known to have many transversals.
Montgomery \etal~\cite{MPS20} recently showed that an $n\times n$
Latin array in which at most $(1-o(1))n$ symbols occur more than
$(1-o(1))n$ times has $(1-o(1))n$ pairwise disjoint transversals.

A partial transversal is \emph{maximal} if it is not contained in any
longer partial transversal. It is not hard to see that a maximal
partial transversal of length $\ell$ in a Latin square of order $n$
must satisfy $n/2\le\ell\le n$. In \cite{BMSW19} it was shown that for
$n\ge5$ all values of $\ell$ in this range are achieved.  Then Evans
\cite{Eva19} constructed an infinite family of Latin squares which
simultaneously have maximal partial transversals of each of the
permissible lengths. Subsequently, Evans \etal~\cite{EMW20}
showed that there exists a Latin square of order $n$ which
has maximal partial transversals of each permissible length if and only if
$n\notin\{3,4\}$ and $n \not\equiv 2 \pmod 4$.

In Latin squares, it is easy to find a partial transversal of length
$\lceil n/2 \rceil$ using a greedy algorithm. A succession of results
have progressively improved on this observation; see \cite{transurvey}
for a history. The three most important breakthroughs have been
the following:

In 1978, Brouwer \etal~\cite{Brouwer78} and
Woolbright~\cite{Woolbright78} independently achieved the first lower
bound that is asymptotically equal to $n$:

\begin{theorem}\label{thm:n-sqrt-n}
Every Latin square of order $n$ contains a partial transversal of
length at least $n-\sqrt{n}$.
\end{theorem}

Correcting an earlier flawed proof, in 2008 Shor and
Hatami~\cite{HatamiShor08} improved the deficit to $O(\log^2 n)$:

\begin{theorem}\label{th:n-log2}
  Every Latin square of order $n$ contains a partial transversal of
  length at least $n-11.053\log^2n$.
\end{theorem}

Finally, very recently Keevash \etal~\cite{KPSY} improved the error
term once again.

\begin{theorem}\label{th:n-log}
  Every Latin square of order $n$ contains a partial transversal of
  length at least $n-O(\log n/\log\log n)$.
\end{theorem}

The goal of this paper is to provide, for $k\le20$, an improved lower
bound on the order $n$ of any Latin array that lacks a partial
transversal of length $n-k$. The case $k=1$ is handled in the next
section, and the case $2\le k\le20$ is addressed in the last section.
Our results rely heavily on computation. Each computation was verified
by at least two independent programs.  Preliminary versions of 
our results were given in the PhD theses of the first two authors
\cite{darcy-thesis,Pula11}.

\section{\label{s:brualdi}Near transversals}

In this section we describe our approach to proving that Latin arrays
of order $n\le11$ have a near transversal. Our method is based on the work
of Shor and Hatami~\cite{HatamiShor08}.
Until the recent breakthrough by Keevash \etal~\cite{KPSY},
\cref{th:n-log2} was the state of the art.
The key idea needed to prove it was the idea of
$\#$-swapping. Consider a diagonal, $T$, of weight $w$. Choose two
entries from $T$, say $(i_0,j_0,k_0)$ and $(i_1,j_1,k_1)$. If
$T\setminus\{(i_0,j_0,k_0),(i_1,j_1,k_1)\}$ still covers $w$ symbols,
then we consider the diagonal
$$\big(T\setminus\{(i_0,j_0,k_0),(i_1,j_1,k_1)\}\big)\cup\big\{(i_0,j_1,\unkn),(i_1,j_0,\unkn)\big\},$$
where we adopt the convention of using $\unkn$ to denote an unknown
symbol (possibly a different symbol each time the notation is used).
This diagonal is guaranteed to have a weight of $w$, $w+1$ or
$w+2$. The act of swapping $\{(i_0,j_0,k_0),(i_1,j_1,k_1)\}$ for
$\{(i_0,j_1,\unkn),(i_1,j_0,\unkn)\}$ to obtain a new diagonal is
called a $\#$-swap. Note that by repeated use of $\#$-swaps, the
weight of the diagonal can never decrease. Thus, if we start with a
diagonal of maximum weight, it is impossible to $\#$-swap to a
diagonal of larger weight and the set of symbols on each diagonal that
we reach by $\#$-swapping will be the same.

Throughout the remainder of the section, we use the symbol \x to
indicate a cell that must contain a symbol that appears on the
original diagonal. For example, if that cell is reachable via a
sequence of $\#$-swaps and the original diagonal has maximum weight,
then the symbol in the cell must appear somewhere on the original
diagonal. In contrast, if a cell is shown as empty, it means that we
do not know anything about it.

\begin{example}
  Here is an example of $\#$-swapping on a diagonal of weight $4 =
  6-2$. If we remove the top left $0$ and $1$ from the diagonal, we
  still have $4$ symbols left, so we may $\#$-swap on these entries
  and instead consider the diagonal which contains the two \x's, with
  the bottom four rows unchanged.
  
  
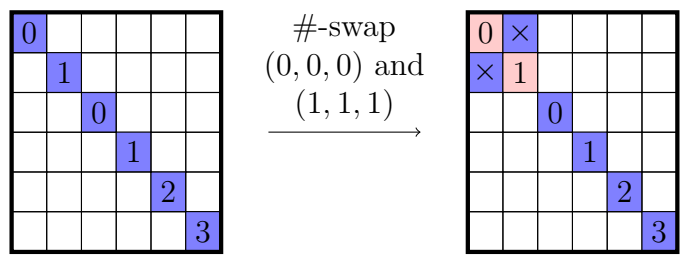
\begin{center}
    \begin{tikzpicture}
      \begin{scope}
        \matrix (A) [square matrix]{
          |[fill=trans]|0 & ~ & ~ & ~ & ~ & ~ \\
          ~ & |[fill=trans]|1 & ~ & ~ & ~ & ~ \\
          ~ & ~  & |[fill=trans]|0 & ~  & ~  & ~ \\
          ~ & ~  & ~  & |[fill=trans]|1 & ~  & ~  \\
          ~ & ~  & ~& ~ & |[fill=trans]|2 & ~ \\
          ~ & ~  & ~&~& ~ & |[fill=trans]|3 \\
        };
        \draw[ultra thick] (A-1-1.north west) rectangle (A-6-6.south east);
      \end{scope}

      \draw[->] (2,0) -- node[text width=2.5cm,midway,above,align=center]{\#-swap $(0,0,0)$ and $(1,1,1)$} (4,0);
      
      \begin{scope}[xshift=6cm]
        \matrix (A) [square matrix]{
          |[fill=trans2]|0 & |[fill=trans]|\x & ~ & ~ & ~ & ~ \\
          |[fill=trans]|\x & |[fill=trans2]|1 & ~ & ~ & ~ & ~ \\
          ~ & ~  & |[fill=trans]|0 & ~  & ~  & ~ \\
          ~ & ~  & ~  & |[fill=trans]|1 & ~  & ~  \\
          ~ & ~  & ~& ~ & |[fill=trans]|2 & ~ \\
          ~ & ~  & ~&~& ~ & |[fill=trans]|3 \\
        };
        \draw[ultra thick] (A-1-1.north west) rectangle (A-6-6.south east);
      \end{scope}
    \end{tikzpicture}
  \end{center}
\end{example}

Note that after performing a $\#$-swap, the two cells that were swapped
out will be lightly shaded for further clarity.

The primary purpose of this paper is to describe a proof of a
generalisation of \cjref{conj:brualdi} for small orders. We try to
find near transversals in all Latin arrays rather than just Latin
squares. We focus on diagonals of weight $n-2$ and attempt to uncover
a new symbol, which would locate a near transversal. The following
elementary observation is needed throughout.

\begin{lemma}\label{l:extend}
  If every Latin array of order $n$ contains a partial transversal of
  length $k$, then every Latin array of order $n+1$ contains a partial
  transversal of length $k$.
\end{lemma}

Throughout the section, we utilise \lref{l:extend} iteratively.
Having shown that all Latin arrays of order $n-1$ contain a near
transversal, we will then know that all Latin arrays of order $n$
contain a diagonal of weight at least $n-2$. A diagonal of weight
$n-2$ has two essentially different configurations for the duplicated
symbols as shown in the following pictures:
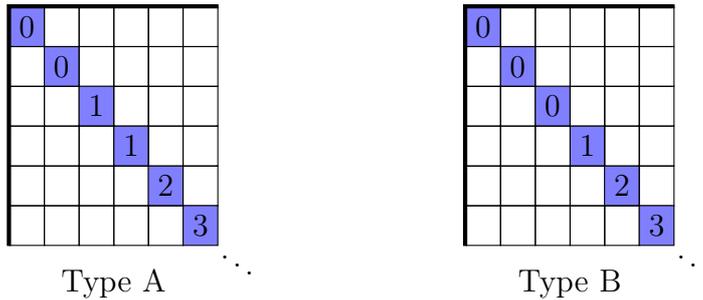
\begin{center}
  \begin{tikzpicture}
    \begin{scope}
      \matrix (A) [square matrix]{
        |[fill=trans]|0 & ~ & ~ & ~ & ~ & ~\\
        ~ & |[fill=trans]|0 & ~ & ~ & ~ & ~\\
        ~ & ~  & |[fill=trans]|1 & ~  & ~  & ~\\
        ~ & ~  & ~  & |[fill=trans]|1 & ~  & ~ \\
        ~ & ~  & ~& ~ & |[fill=trans]|2 & ~\\
        ~ & ~  & ~&~& ~ & |[fill=trans]|3 \\
      };
      \draw[ultra thick] (A-1-6.north east) -- (A-1-1.north west) -- (A-6-1.south west);
      \node at ([xshift=0.1cm]A.south east) {$\ddots$};

      \node at ([yshift=-10pt]A.south) {Type A};
    \end{scope}
    
    \begin{scope}[xshift=6cm]
      \matrix (A) [square matrix]{
        |[fill=trans]|0 & ~ & ~ & ~ & ~ & ~\\
        ~ & |[fill=trans]|0 & ~ & ~ & ~ & ~\\
        ~ & ~  & |[fill=trans]|0 & ~  & ~  & ~\\
        ~ & ~  & ~  & |[fill=trans]|1 & ~  & ~ \\
        ~ & ~  & ~& ~ & |[fill=trans]|2 & ~\\
        ~ & ~  & ~&~& ~ & |[fill=trans]|3 \\
      };
      \draw[ultra thick] (A-1-6.north east) -- (A-1-1.north west) -- (A-6-1.south west);
      \node at ([xshift=0.1cm]A.south east) {$\ddots$};
      
      \node at ([yshift=-10pt]A.south) {Type B};
    \end{scope}
  \end{tikzpicture}
\end{center}
A diagonal of type A has two symbols which each occur twice on the
diagonal, whilst a diagonal of type B has one symbol that occurs
thrice on the diagonal.  We first start by showing that the existence
of a type B diagonal implies the existence of a type A diagonal in
maximal cases.

\begin{lemma}\label{lem:weak-trans}
Let $L$ be a Latin array of order $n$ with a diagonal of weight $n-2$
and no diagonal of weight greater than $n-2$. If $L$ has a diagonal of
type B, then there exists a diagonal of type A that can be reached by
a sequence of $\#$-swaps.
\end{lemma}

\begin{proof}
Assume, on the contrary, that no diagonal of type A can be
reached. Without loss of generality, the initial diagonal of type B is
the main diagonal and the three repeated symbols are in the top 3
rows.
    
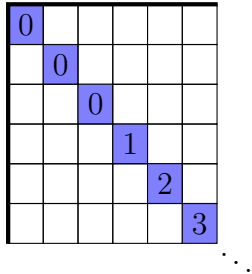
\begin{center}
  \begin{tikzpicture}
    \begin{scope}
      \matrix (A) [square matrix]{
        |[fill=trans]|0 & ~ & ~ & ~ & ~ & ~\\
        ~ & |[fill=trans]|0 & ~ & ~ & ~ & ~\\
        ~ & ~  & |[fill=trans]|0 & ~  & ~  & ~\\
        ~ & ~  & ~  & |[fill=trans]|1 & ~  & ~ \\
        ~ & ~  & ~& ~ & |[fill=trans]|2 & ~\\
        ~ & ~  & ~&~& ~ & |[fill=trans]|3 \\
      };
      \draw[ultra thick] (A-1-6.north east) -- (A-1-1.north west) -- (A-6-1.south west);

      \node at ([xshift=0.1cm]A.south east) {$\ddots$};
    \end{scope}
  \end{tikzpicture}
\end{center}

We need to perform $n-2$ $\#$-swaps to arrive at a contradiction. 
First, we $\#$-swap $(0,0,0)$ and $(2,2,0)$. The symbols in the cells
$(0,2)$ and $(2,0)$ must be the same, otherwise this new diagonal
would be of type A. Without loss of generality, these cells contain
the symbol 1. We now $\#$-swap $(0,2,1)$ and $(3,3,1)$. By a similar
argument, the two uncovered cells must contain the same symbols (which
is, without loss of generality, 2). We repeat this same argument $n-2$
times in total. On all steps $i$ (except the first one), we $\#$-swap
the entries $(0,i,i-1)$ and $(i+1,i+1,i-1)$ and expose the entries
$(0,i+1,i)$ and $(i+1,i,i)$. The first three steps are shown in
\cref{fig:steps-weak-trans}.

\begin{figure}[h]
  \centering
  \begin{tikzpicture}
    \begin{scope}
      \matrix (A) [square matrix]{
        |[fill=trans]|0 & ~ & ~ & ~ & ~ & ~\\
        ~ & |[fill=trans]|0 & ~ & ~ & ~ & ~\\
        ~ & ~  & |[fill=trans]|0 & ~  & ~  & ~\\
        ~ & ~  & ~  & |[fill=trans]|1 & ~  & ~ \\
        ~ & ~  & ~& ~ & |[fill=trans]|2 & ~\\
        ~ & ~  & ~&~& ~ & |[fill=trans]|3 \\
      };
      \draw[ultra thick] (A-1-6.north east) -- (A-1-1.north west) -- (A-6-1.south west);
      
      \node at ([xshift=0.1cm]A.south east) {$\ddots$};
    \end{scope}
    
    \begin{scope}[xshift=4cm]
      \matrix (A) [square matrix]{
        |[fill=trans2]|0 & ~ & |[fill=trans]|1 & ~ & ~ & ~\\
        ~ & |[fill=trans]|0 & ~ & ~ & ~ & ~\\
        |[fill=trans]|1 & ~  & |[fill=trans2]|0 & ~  & ~  & ~\\
        ~ & ~  & ~  & |[fill=trans]|1 & ~  & ~ \\
        ~ & ~  & ~& ~ & |[fill=trans]|2 & ~\\
        ~ & ~  & ~&~& ~ & |[fill=trans]|3 \\
      };
      \draw[ultra thick] (A-1-6.north east) -- (A-1-1.north west) -- (A-6-1.south west);

      \node at ([xshift=0.1cm]A.south east) {$\ddots$};
    \end{scope}
    
    \begin{scope}[xshift=8cm]
      \matrix (A) [square matrix]{
        0 & ~ & |[fill=trans2]|1 & |[fill=trans]|2 & ~ & ~\\
        ~ & |[fill=trans]|0 & ~ & ~ & ~ & ~\\
        |[fill=trans]|1 & ~  & 0 & ~  & ~  & ~\\
        ~ & ~  & |[fill=trans]|2 & |[fill=trans2]|1 & ~  & ~ \\
        ~ & ~  & ~& ~ & |[fill=trans]|2 & ~\\
        ~ & ~  & ~&~& ~ & |[fill=trans]|3 \\
      };
      \draw[ultra thick] (A-1-6.north east) -- (A-1-1.north west) -- (A-6-1.south west);

      \node at ([xshift=0.1cm]A.south east) {$\ddots$};
    \end{scope}
    
    \begin{scope}[xshift=12cm]
      \matrix (A) [square matrix]{
        0 & ~ & 1 & |[fill=trans2]|2 & |[fill=trans]|3 & ~\\
        ~ & |[fill=trans]|0 & ~ & ~ & ~ & ~\\
        |[fill=trans]|1 & ~  & 0 & ~  & ~  & ~\\
        ~ & ~  & |[fill=trans]|2  & 1 & ~  & ~ \\
        ~ & ~  & ~& |[fill=trans]|3 & |[fill=trans2]|2 & ~\\
        ~ & ~  & ~&~& ~ & |[fill=trans]|3 \\
      };
      \draw[ultra thick] (A-1-6.north east) -- (A-1-1.north west) -- (A-6-1.south west);

      \node at ([xshift=0.1cm]A.south east) {$\ddots$};
    \end{scope}
  \end{tikzpicture}
  \caption{\label{fig:steps-weak-trans}The first three $\#$-swaps in \cref{lem:weak-trans}.}
\end{figure}
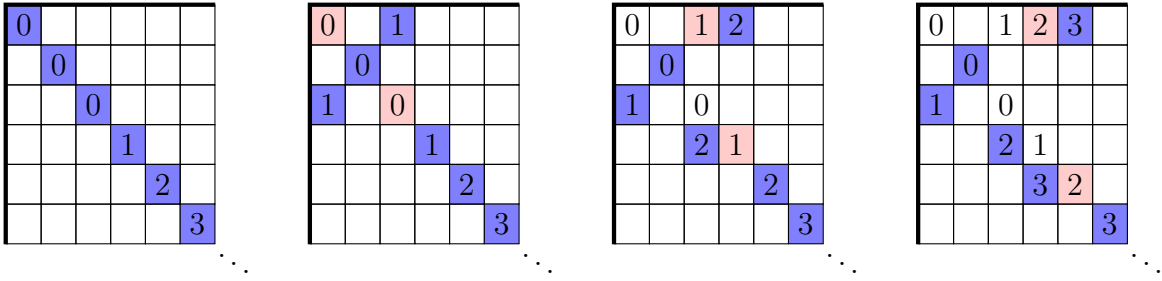

However, at step $n-2$, the uncovered symbol must be some symbol that
did not appear on the original diagonal. Thus, we have found a heavier
diagonal, a contradiction.
\end{proof}

Our next result generalises \Cref{lem:weak-trans}, except that we
abandon the condition that we must be able to $\#$-swap to the new
diagonal. It also generalises \cite[Prop.7]{CW05}, whose proof it
mimics.  Note that \cite[Prop.7]{CW05} has been generalised in a
different direction (namely, to row-Latin arrays of order $n$
containing $n$ symbols) by Aharoni \etal~\cite{ABKZ18}.

\begin{lemma}\label{lem:partial-to-weak}
Any entry of a Latin array contained in a diagonal of weight $w$ is
contained in a diagonal of weight at least $w$ where each symbol
appears on the diagonal at most twice.
\end{lemma}

\begin{proof}
Let $L$ be a Latin array of order $n$. For convenience, we will assume
that the diagonal in question is the main diagonal and let $M$ be the
multiset of symbols on the main diagonal. If no symbol appears more
than twice in $M$, we are done. Otherwise, fix some entry
$(r,r,\unkn)$. We will find a diagonal of weight at least $w$ with the
desired properties that still contains $(r,r,\unkn)$.
 
Select $r_1 \neq r$ such that the symbol $L(r_1,r_1)$ appears in $M$
three or more times. Let $x_i$ be the number of symbols appearing
exactly $i$ times in $M$. It follows that
$$\sum_{i=0}^n ix_i = n.$$
Thus, 
$n-(x_1+x_2+\cdots+x_n) = x_2 + 2x_3 + \cdots + (n - 1)x_n > x_2 + x_3 + \cdots + x_n$.

Row $r_1$ contains at least $n-(x_1+x_2+\cdots+x_n)$ 
symbols that do not appear in $M$ and column $r_1$ contains $(\sum_{i=2}^n
x_i)-1$ symbols that appear more than once in $M$, besides
the entry $(r_1, r_1,\unkn)$. Thus, there are at least two values of
$r_2$ such that the symbol $L(r_1, r_2)$ does not appear in $M$ and
the symbol $L(r_2, r_1)$ does not appear more than once in $M$. Select
$r_2 \neq r$. Observe that 
$$T = \big(M \setminus\{(r_1,r_1,\unkn),(r_2,r_2,\unkn)\}\big) 
\cup \big\{(r_1,r_2,\unkn),(r_2,r_1,\unkn)\big\}$$ 
has fewer cells than $M$ that contain symbols that
appear more than twice in the diagonal and $(r,r,\unkn)$ still belongs
to $T$. Furthermore, $T$ is of weight at least $w$. By iterating,
we will therefore find a diagonal with the desired properties.
\end{proof}

Next, we give a simple example of how using $\#$-swaps is useful.

\begin{lemma}\label{lem:ord-6-near}
  In any Latin array of order $6$, there exists a diagonal of weight
  at least $5$.
\end{lemma}

\begin{proof}
First, it is quite easy to show that the heaviest diagonal must be at
least of weight 4 (for example, use \cref{thm:n-sqrt-n}). We now
assume, on the contrary, that there exists a Latin array that contains
a diagonal of weight $4$, but none of weight 5 or 6. Without loss of
generality, the original diagonal of length $4$ is along the main
diagonal. By \cref{lem:weak-trans}, we may assume that it takes the form
given here.
    
\begin{center}
  \begin{tikzpicture}
    \begin{scope}
      \matrix (A) [square matrix]{
        |[fill=trans]|0 & ~ & ~ & ~ & ~ & ~ \\
        ~ & |[fill=trans]|0 & ~ & ~ & ~ & ~ \\
        ~ & ~  & |[fill=trans]|1 & ~  & ~  & ~ \\
        ~ & ~  & ~  & |[fill=trans]|1 & ~  & ~  \\
        ~ & ~  & ~& ~ & |[fill=trans]|2 & ~ \\
        ~ & ~  & ~&~& ~ & |[fill=trans]|3 \\
      };
      \draw[ultra thick] (A-1-1.north west) rectangle (A-6-6.south east);
    \end{scope}
  \end{tikzpicture}
\end{center}
    
At this point, we are presented with four options for which pair of
entries to $\#$-swap (choose either $0$ and either $1$
independently). From this, we can see that we have the following.

\begin{center}
  \begin{tikzpicture}
    \begin{scope}
      \matrix (A) [square matrix]{
        |[fill=trans]|0 & ~ & \x & \x & ~ & ~ \\
        ~ & |[fill=trans]|0 & \x & \x & ~ & ~ \\
        \x & \x  & |[fill=trans]|1 & ~  & ~  & ~ \\
        \x & \x  & ~  & |[fill=trans]|1 & ~  & ~  \\
        ~ & ~  & ~& ~ & |[fill=trans]|2 & ~ \\
        ~ & ~  & ~&~& ~ & |[fill=trans]|3 \\
      };
      \draw[ultra thick] (A-1-1.north west) rectangle (A-6-6.south east);
    \end{scope}
  \end{tikzpicture}
\end{center}

As explained above, each \x must be one of $0,\,1,\,2,\,3$; otherwise, we
would have a heavier diagonal. Consider $\#$-swapping the entries
in the first row and the third row. The symbol in the $(2,0)$ cell
must be either $2$ or $3$. Without loss of generality, we assume that
it is a $2$.

\begin{center}
  \begin{tikzpicture}
    \begin{scope}
      \matrix (A) [square matrix]{
        |[fill=trans2]|0 & ~ & |[fill=trans]|\x & \x & ~ & ~ \\
        ~ & |[fill=trans]|0 & \x & \x & ~ & ~ \\
        |[fill=trans]|2 & \x  & |[fill=trans2]|1 & ~  & ~  & ~ \\
        \x & \x  & ~  & |[fill=trans]|1 & ~  & ~  \\
        ~ & ~  & ~& ~ & |[fill=trans]|2 & ~ \\
        ~ & ~  & ~&~& ~ & |[fill=trans]|3 \\
      };
      \draw[ultra thick] (A-1-1.north west) rectangle (A-6-6.south east);
    \end{scope}
  \end{tikzpicture}
\end{center}
        
Note that we do not know what symbol is in the $(0,2)$ cell, but we do
know that it is a duplicate symbol (\ie~it appears at least one more
time on the diagonal or at least two more times if it is a $2$). Thus,
we are free to $\#$-swap on that entry now. We $\#$-swap that entry
and $(4,4,2)$.
    
\begin{center}
  \begin{tikzpicture}
    \begin{scope}
      \matrix (A) [square matrix]{
        0 & ~ & |[fill=trans2]|\x & \x & |[fill=trans]|\x & ~ \\
        ~ & |[fill=trans]|0 & \x & \x & ~ & ~ \\
        |[fill=trans]|2 & \x  & 1 & ~  & ~  & ~ \\
        \x & \x & ~ & |[fill=trans]|1 & ~  & ~  \\
        ~ & ~ & |[fill=trans]|\x & ~ & |[fill=trans2]|2 & ~ \\
        ~ & ~ & ~ &~& ~ & |[fill=trans]|3 \\
      };
      \draw[ultra thick] (A-1-1.north west) rectangle (A-6-6.south east);
    \end{scope}
  \end{tikzpicture}
\end{center}
    
The symbol in the $(4,2)$ cell must be either $0$ or $3$ and the
symbol in the $(0,4)$ cell is a duplicate (as described above), and so
may be used immediately. At this point, we consider both cases for the
$(4,2)$ cell separately. In either case, we $\#$-swap the entry in the
top row with the appropriate duplicated symbol.
\begin{center}
  \begin{tikzpicture}
    \begin{scope}
      \matrix (A) [square matrix]{
        0 & |[fill=trans]|\x & \x & \x & |[fill=trans2]|\x & ~ \\
        ~ & |[fill=trans2]|0 & \x & \x & |[fill=trans]|\x & ~ \\
        |[fill=trans]|2 & \x  & 1 & ~  & ~  & ~ \\
        \x & \x & ~ & |[fill=trans]|1 & ~  & ~  \\
        ~ & ~ & |[fill=trans]|0 & ~ & 2 & ~ \\
        ~ & ~ & ~ &~& ~ & |[fill=trans]|3 \\
      };
      \draw[ultra thick] (A-1-1.north west) rectangle (A-6-6.south east);
    \end{scope}
    
    \begin{scope}[xshift=4cm]
      \matrix (A) [square matrix]{
        0 & ~ & \x & \x & |[fill=trans2]|\x & |[fill=trans]|\x \\
        ~ & |[fill=trans]|0 & \x & \x & ~ & ~ \\
        |[fill=trans]|2 & \x  & 1 & ~  & ~  & ~ \\
        \x & \x & ~ & |[fill=trans]|1 & ~  & ~  \\
        ~ & ~ & |[fill=trans]|3 & ~ & 2 & ~ \\
        ~ & ~ & ~ &~& |[fill=trans]|\x & |[fill=trans2]|3 \\
      };
      \draw[ultra thick] (A-1-1.north west) rectangle (A-6-6.south east);
    \end{scope}
  \end{tikzpicture}
\end{center}

In either case, the top row now has five entries whose symbol must
come from the set $\{0,1,2,3\}$, which is impossible in a Latin array.
The result follows.
\end{proof}

Hatami and Shor~\cite{HatamiShor08} used this same idea to show 
\cref{lem:ord-6-near}. However, in their description,
they did not leave all of the symbols in the top row as unknown
(\x). Instead, they did extra case analysis to determine what those
symbols could be. By leaving the top row as unknown symbols, there is
the potential for less branching in the algorithm. Moreover, by
continually using the top row to $\#$-swap on, there are only two
choices of pairs of entries to $\#$-swap (but one of these choices
undoes the last change and reverts to the previous diagonal).

We describe two algorithms whose goal is to show that there is a near
transversal in all Latin arrays of a particular order
$n$. \Cref{alg:hash1} describes the basic algorithm to show that all
Latin arrays of order $n$ contain a near transversal. This algorithm
formalises the method used in the proof of \cref{lem:ord-6-near}. This
algorithm is then refined in \cref{alg:hash2}, which succeeds for larger
orders than \cref{alg:hash1} does.

It is important to note that in both algorithms below, all variables
are considered local variables, so changing the value of a parameter
does not affect its value outside of that specific instance.

\begin{algorithm}
  \caption{Basic algorithm to show that all Latin arrays of order $n$ contain a near transversal. 
    \textsc{NaiveHash}$(L,\epsilon,0,3)$ should be called initially, where $L$ is an $n \times n$ ($n \geq 4$) array with all cells empty except the main diagonal, which contains $(0,0,1,1,2,3,\dots,n-3)$ and $\epsilon$ is the identity permutation. Note that 3 is an arbitrary choice---we could have selected 2 or 3 (the rows of the duplicated symbol 1). }
  
  \label{alg:hash1}
  \hspace*{\algorithmicindent} \textbf{Input} $L$ is a partial Latin array with some otherwise empty cells marked with $\X$\\
  \hspace*{\algorithmicindent} \textbf{Input} $\sigma$ is a permutation defining a diagonal of weight $n-2$ in $L$\\
  \hspace*{\algorithmicindent} \textbf{Input} $d$ is the depth of the search\\
  \hspace*{\algorithmicindent} \textbf{Input} $r$ is the row we just hashed on\\
  \hspace*{\algorithmicindent} \textbf{Output} \textsc{True} if every Latin array that is a completion of the input has a near transversal.\\
  \hspace*{\algorithmicindent} \textbf{Output} \textsc{False} if the computation is inconclusive.

  \begin{algorithmic}[1]
    
    \Procedure{NaiveHash}{$L,\sigma,d,r$}

    \If {Some row or column of $L$ contains at least $n-1$ filled cells}
      \State\Return{\textsc{True}}\myComment{Near transversal guaranteed}
    \EndIf
    \If {$d \neq 0$ and $\sigma$ is the identity and $r = 3$}
      \State\Return{\textsc{False}}\label{algline:false1}\myComment{We have cycled back to where we started}
    \EndIf

    \State~
    \State $S \gets L(r,\sigma_r)$ \myComment{Symbol to hash on}
    \State $R \gets $ \text{row such that } $\sigma_R = S$, $R > 0$ \text{ and } $R \neq r$ \myComment{Other row that contains $S$ on $\sigma$}

    \State swap$(\sigma_0,\sigma_R)$ \myComment{Update $\sigma$ to enact the $\#$-swap}
    \State ~
    \If{$L(R,\sigma_R) \neq \X$} \myComment{If we already know what symbol this is}
      \State \Return{\Call{NaiveHash}{$L,\sigma,d+1,R$}}
    \Else \myComment{If we do not know what symbol is here, then try all valid ones.}
      \State $k \gets $ largest symbol in $L$ that appears multiple times
      \State \myComment{The symbols $k+1, k+2, \dots, n-3$ are all symmetric up to} \\ \hspace*{\fill} {\color{blue} this point, so we only need to consider one of them} \\ \hspace*{\fill} {\color{blue} (without loss of generality, we use $k+1$).}
      \For{$s \gets 0$ to min($k+1,n-3$)}
        \If{$s$ is not in row $R$ nor column $\sigma_R$}
          \State $L(R,\sigma_R) \gets s$
          \If{\Call{NaiveHash}{$L,\sigma,d+1,R$} = \textsc{False}}
            \State\Return{\textsc{False}}
          \EndIf
        \EndIf
       \EndFor

       \State \Return{\textsc{True}} \myComment{No matter which symbol we use, there is a near transversal}
    \EndIf
    \EndProcedure
    \end{algorithmic}
\end{algorithm}

\Cref{alg:hash1} is sufficient to show that all Latin arrays of order
$n \leq 7$ contain a near transversal (this result was obtained by an
independent calculation in \cite{BHWWW18}). However, for $n = 8$,
\cref{alg:hash1} fails to show the desired result as it returns False
for \cref{fig:fail-alg1}.

\begin{figure}[h]
\centering
  \begin{tikzpicture}
    \begin{scope}
      \matrix (A) [square matrix]{
        0 & \x & \x & \x & \x & \x &  ~ &  ~ \\
        ~ &  0 &  ~ &  ~ &  1 &  2 &  ~ &  ~ \\
        2 &  ~ &  1 &  3 &  ~ &  ~ &  ~ &  ~ \\
        3 &  2 &  ~ &  1 &  0 &  ~ &  ~ &  ~ \\
        ~ &  3 &  0 &  ~ &  2 &  ~ &  ~ &  ~ \\
        1 &  ~ &  ~ &  0 &  ~ &  3 &  ~ &  ~ \\
        ~ &  ~ &  ~ &  ~ &  ~ &  ~ &  4 &  ~ \\
        ~ &  ~ &  ~ &  ~ &  ~ &  ~ &  ~ &  5 \\
      };
      \draw[ultra thick] (A-1-1.north west) rectangle (A-8-8.south east);
    \end{scope}
  \end{tikzpicture}
  \caption{\label{fig:fail-alg1}One of 14 squares that fail Algorithm 1 for $n=8$.}
\end{figure}

A total of 14 partial Latin arrays fail \cref{alg:hash1} for $n=8$
(meaning that Line 5 of \cref{alg:hash1} is reached 14 times). For
$n=9$, one may expect more squares to fail \cref{alg:hash1}, but
interestingly, those 14 squares (with one extra row and column added)
are the only squares to fail \cref{alg:hash1}. For $n=10$, a total of
82\,140 squares fail \cref{alg:hash1}. By comparison, Line 3 of
\cref{alg:hash1} was reached 2\,657, 377\,452 and 696\,808\,457 times
respectively for $n=8,\,9,\,10$.

The failures such as \cref{fig:fail-alg1} show that a more refined
approach is needed to find near transversals in larger orders. The
first observation is that after we have cycled back on ourselves and
returned False on Line~\ref{algline:false1} of \cref{alg:hash1}, we
may now choose another row to $\#$-swap on, rather than the first
one. Recall that by only using $\#$-swaps on the top row, we are only
utilising two of the possible $\#$-swaps available (there are 4
possible if the diagonal is of type A and 3 if it is of type B). We
also note in passing that, although our searches are always based
around the main diagonal, there are other options available as soon as
other diagonals of weight $n-2$ have been located.

In \cref{alg:hash2}, we first $\#$-swap along the top
row. Once we cycle around, we then $\#$-swap along the second row,
then the third, then the fourth. In \cref{alg:hash1}, we arbitrarily
selected row 3 to be the initial value for $r$. In \cref{alg:hash2},
when we are $\#$-swapping on rows 0, 1, 2 and 3, we use the rows 3, 2, 1
and 0, respectively for the initial value of the ``row we just
$\#$-swapped on''. It was convenient, but far from essential, to know
that $r_0+r_1=3$.

There are two heuristics that can be added to the search that
significantly improve its performance when utilised
together. (However, there is a minor drawback to using them, which we
will discuss in \sref{s:nearnear}.) The first heuristic is to search
for diagonals of weight $n-2$ that may not be reachable via
$\#$-swaps. If a partial transversal of length $n-2$ covers
all rows except $r_0$ and $r_1$
and all columns except $c_0$ and $c_1$, then we know that each of the
cells $(r_0,c_0), (r_0,c_1), (r_1,c_0)$ and $(r_1,c_1)$ must also
contain symbols from $\{0,1,\dots,n-3\}$, or else we would have a near
transversal. Thus, if those cells are empty, we may fill them with an
\x. In practice, every time that we fill a cell with a specific symbol
(not an \x), we only search for diagonals that go through that
cell. The second heuristic is to choose some \x in the square and
decide what that symbol should be by exhaustively trying each one. We
define the \emph{liberties} of a cell to be the number of symbols that
could be placed into the cell without violating the Latin property. In
practice, we choose an \x that has the fewest liberties to limit the
branching that our search does. These heuristics are combined in
\cref{alg:hash2}.

\begin{algorithm}
  \caption{More advanced algorithm to determine if all Latin arrays of order $n$ contain a near transversal.
    \textsc{Hash}$(L,\epsilon,0,0,3)$ should be called initially, where $L$ is an $n \times n$ ($n \geq 4$) array with all cells empty except the main diagonal, which contains $(0,0,1,1,2,3,\dots,n-3)$ and $\epsilon$ is the identity permutation. \textsc{FillCell} simply tries all valid symbols to place in the cell $(r,c)$ and calls \textsc{Hash} with the same parameters, but with depth $d+1$.
  }
  \label{alg:hash2}
  \hspace*{\algorithmicindent} \textbf{Input} $L$ is a partial Latin array with some otherwise empty cells marked with $\X$\\
  \hspace*{\algorithmicindent} \textbf{Input} $\sigma$ is a permutation defining a diagonal of weight $n-2$\\
  \hspace*{\algorithmicindent} \textbf{Input} $d$ is the depth of the search with the current $r_0$\\
  \hspace*{\algorithmicindent} \textbf{Input} $r_0$ is the row we are mainly $\#$-swapping on (this was the top row in \cref{alg:hash1})\\
  \hspace*{\algorithmicindent} \textbf{Input} $r_1$ is the other row that we just $\#$-swapped on\\
  \hspace*{\algorithmicindent} \textbf{Output} \textsc{True} if every Latin array that is a completion of the input has a near transversal.\\
  \hspace*{\algorithmicindent} \textbf{Output} \textsc{False} if inconclusive.

  \begin{algorithmic}[1]
    
    \Procedure{Hash}{$L,\sigma,d,r_0,r_1$}

    \If {Some row or column of $L$ contains at least $n-1$ filled cells}
      \State\Return{\textsc{True}}\myComment{Near transversal guaranteed}
    \EndIf

    \If {$d \neq 0$ and $\sigma$ is the identity and $r_0+r_1 = 3$} \label{algline:loopback2}\myComment{We have cycled back}
      \If {$r_0 \geq 3$} \Return{\textsc{False}}\label{algline:false2}\myComment{We need to try something different.}
      \Else { \Return{\Call{Hash}{$L,\sigma,0,r_0+1,r_1-1$}} \myComment{Try $\#$-swapping along the next row.} }
      \EndIf
    \EndIf

    \State~
    
    \If{$d \equiv 3 \pmod 4$ and there is at least one \x in $L$}
      \State $(r,c) \gets $ cell such that $L(r,c) = \X$.
      \myComment{If there are multiple \x, select one with the} \\ \hspace*{\fill} {\color{blue} fewest liberties, breaking ties by selecting} \\ \hspace*{\fill} {\color{blue} the first one in row-major order.}
      \State\Return{\Call{FillCell}{$L,r,c,\sigma,d,r_0,r_1$}}
    \EndIf

    \State~
    \State $S \gets L(r_1,\sigma_{r_1})$ \myComment{Symbol to $\#$-swap on}
    \State $R \gets $ \text{row where } $\sigma_R = S$, $R \neq r_0$ \text{ and } $R \neq r_1$ \myComment{Other row that contains $S$ on $\sigma$}

    \State swap$(\sigma_{r_0},\sigma_R)$ \myComment{Update $\sigma$ to enact the $\#$-swap}
    
    \State\Return{\Call{FillCell}{$L,R,\sigma_R,\sigma,d,r_0,R$}}
    
    \EndProcedure

    \State~
    \Procedure{FillCell}{$L,r,c,\sigma,d,r_0,r_1$}
      \If{$L(r,c) \in \{0,\dots,n-3\}$} \Return{\Call{Hash}{$L,\sigma,d+1,r_0,r_1$}} \EndIf

      \State~

      \State $k \gets $ largest symbol in $L$ that appears multiple times
      \For{$s \gets 0$ to min($k+1,n-3$)}
         \If{$s$ is not in row $r$ nor column $c$}
          \State $L' \gets L$ \myComment{Store a copy of $L$}
          
          \State $L(r,c) \gets s$
          \ForEach {Partial transversal, $T$, of length $n-2$}
            \State $\{R_1,R_2,C_1,C_2\} \gets$ the two rows and two columns missing from $T$
            \State Fill in cells $(R_1,C_1), (R_1,C_2), (R_2,C_1), (R_2,C_2)$ with \x if empty
          \EndFor
          \If{\Call{Hash}{$L,\sigma,d+1,r_0,r_1$} = \textsc{False}}
            \State\Return{\textsc{False}}
          \EndIf
          \State $L \gets L'$ \myComment{Restore $L$ to its previous configuration}
        \EndIf
       \EndFor

       \State \Return{\textsc{True}} \myComment{No matter which symbol we place here, there is a near transversal}

    \EndProcedure
    \end{algorithmic}
\end{algorithm}

\Cref{alg:hash2} is good enough to show the following result. Both
implementations of the algorithm reached Line 5 the same number of
times (namely 53 times for $n=10$ and $105\,287$ times for $n=11$ with
$r_0=0$ in all cases) providing some corroboration of each other.

\begin{theorem}\label{th:near-trans-11}
  Every Latin array of order $n \leq 11$ contains a near transversal.
\end{theorem}

\begin{proof}
  The cases where $n < 4$ are easy to see (they also follow from
  \cref{thm:n-sqrt-n}). We iteratively use \lref{l:extend} and
  \cref{alg:hash2} for $n=4,\dots,11$.
\end{proof}

In the search for $n = 11$, every \#-swap included one of the first two
rows (that is, $r_0 \in \{0,1\}$ in \cref{alg:hash2}).  The search for
$n \leq 10$ can be completed in a matter of minutes, while a few hours
is needed for $n=11$. Based on this progression, we initially believed
$n = 12$ to be possible. However, after running our program for
several months with several different pruning heuristics on a small
grid, we did not think that the program would finish in a reasonable
amount of time. Due to the recursive nature of the algorithm, it is
difficult to accurately determine what percentage of the search space
was covered over those months. Needless to say, all cases that we
searched did not provide a counterexample to \cjref{conj:brualdi}.

If one wishes to use \cref{alg:hash2} to find near transversals in
Latin squares (rather than in all Latin arrays), then we would
recommend extra heuristics be employed. For example, the partial Latin
array in \cref{fig:loop-back-sqr} is one of many squares that loops
back on \cref{algline:loopback2} of \cref{alg:hash2} with $r_0=0$. The
search continues with $r_0 = 1$, however, no further search is needed
if we are only concerned with Latin squares. There are only two
symbols which are not on the original diagonal (9 and 10), but there
are not enough empty cells left to place them into. In particular, at
least $2(n-2)-4$ cells in the top $(n-2) \times (n-2)$ submatrix must
be empty in order to fit in the two missing symbols.

\begin{figure}
  \centering
  \begin{tikzpicture}
    \begin{scope}
      \matrix (A) [square matrix]{
         0 &  6 & \x &  5 &  3 & \x &  2 &  4 &  1 &  ~ &  ~ \\
        \x &  0 &  3 & \x &  1 & \x & \x & \x & \x &  ~ &  ~ \\
         2 & \x &  1 &  4 & \x & \x & \x &  0 &  3 &  ~ &  ~ \\
         6 &  2 & \x &  1 & \x & \x & \x & \x &  5 &  ~ &  ~ \\
        \x & \x &  0 &  3 &  2 & \x & \x &  6 &  4 &  ~ &  ~ \\
        \x &  1 &  5 & \x & \x &  3 &  6 & \x & \x &  ~ &  ~ \\
         3 & \x & \x & \x &  5 & \x &  4 & \x & \x &  ~ &  ~ \\
        \x & \x & \x & \x & \x &  2 &  1 &  5 &  0 &  ~ &  ~ \\
        \x & \x & \x &  2 & \x &  5 &  0 & \x &  6 &  ~ &  ~ \\
         ~ &  ~ &  ~ &  ~ &  ~ &  ~ &  ~ &  ~ &  ~ &  7 &  ~ \\
         ~ &  ~ &  ~ &  ~ &  ~ &  ~ &  ~ &  ~ &  ~ &  ~ &  8 \\
      };
      \draw[ultra thick] (A-1-1.north west) rectangle (A-11-11.south east);
    \end{scope}
  \end{tikzpicture}
  \caption{\label{fig:loop-back-sqr} A partial Latin array that loops back in \cref{algline:loopback2} when $r_0=0$.}
\end{figure}

The fact that all Latin arrays, and not just Latin squares, have near
transversals is an encouraging sign for \cjref{conj:brualdi}. It is
plausible that an even stronger result holds:

\begin{conjecture}\label{conj:general-brualdi}
  Every Latin array contains a near transversal.
\end{conjecture}

This conjecture is implied by \Cref{conj:brualdi-rectangle}.  Of
course, the fact that our arrays are Latin seems to be a very
important factor in the potential truth of either conjecture. The idea
used in \cref{th:stein-wrong} relies heavily on clumping all $n$ of
each symbol into a small submatrix. The idea cannot be easily changed
to accommodate only one of each symbol per row or
column. \tref{th:stein-wrong} was not optimised in \cite{PS19}. It
would be interesting to know the smallest value of $n$ where an
equi-$n$-square exists that does not contain a near transversal.

\section{\label{s:nearnear}Long partial transversals}

In this final section, we improve the lower bounds on the length of
the longest partial tranversal that Latin arrays of all orders $n<1449$
are known to possess. In the proof of \cref{th:n-log2} by Shor and
Hatami~\cite{HatamiShor08}, one of the key ingredients was sets of
diagonals with the same weight that were connected by a sequence of
$\#$-swaps. A sequence of integers $n_k$ was discussed in detail. To
connect those to our results here, $n_2$ is defined as the smallest
order such that a diagonal of weight $n-2$ cannot be $\#$-swapped to
uncover a new symbol. Note that the heuristics employed in
\cref{alg:hash2} mean that we cannot use the results from
\cref{th:near-trans-11} to show that $n_2 \geq 12$. However, we
verified that $n_2 \geq 11$ utilising a similar idea to
\cref{alg:hash1,alg:hash2}.

Upon first glance, the $n-11.053\log^2n$ bound shown by Shor and
Hatami~\cite{HatamiShor08} is weaker than the $n-\sqrt{n}$ bound in
\tref{thm:n-sqrt-n} for small values of $n$. In fact, for
$n\leq7\,731\,462$, it is better to use the $n-\sqrt{n}$
bound. However, the groundwork laid out in the asymptotic proof in
\cite{HatamiShor08} can be used in a concrete way to show
significantly better bounds for lower orders. The key sequence, $n_k$,
is a bound on the size that a square must have before being able to
\#-swap from a diagonal of weight $n-k$ to a heavier one. In
particular, any Latin array of order $n < n_k$ contains a diagonal
with weight greater than $n-k$.

The following lemma is taken from \cite{HatamiShor08}, except the
first inequality has been strengthened as explained above.

\begin{lemma}\label{lem:shor-eq}
  \begin{equation} \label{eq:shor1} n_2 \geq 11, \end{equation}
  \begin{equation} \label{eq:shor2} n_k \geq n_{k-1} + 2k \quad \text{ for } k > 2 \text { and } \end{equation}
  \begin{equation} \label{eq:shor3} (n_k-n_j)(2n_j+n_{k-1}-2n_k+2k-j) \leq n_j(n_j-n_{j-1}-2j) \quad \text{ for } 3 \leq j < k. \end{equation}
\end{lemma}

Shor and Hatami used \cref{eq:shor3} to show that $k \leq 11.053
\log^2n_k$. While this seems worse than \cref{thm:n-sqrt-n} for small
values, simple induction using \cref{eq:shor1,eq:shor2} shows that
$n_k \ge k^2+k+5$, giving a better bound than \cref{thm:n-sqrt-n} for all
$n$. For small values, the value of $11.053\log^2 n$ is far from
the truth.

Searching for a single sequence that satisfies inequalities
\cref{eq:shor1,eq:shor2,eq:shor3} is quite simple. In fact, for any
sequence $[n_2, \dots, n_\ell]$ that satisfies
\cref{eq:shor1,eq:shor2,eq:shor3}, you may extend it by setting
$n_{\ell+1} = 2n_\ell + 2(\ell+1)$ and it will still satisfy
\cref{eq:shor1,eq:shor2,eq:shor3}. However, the true interest lies in
the smallest value that $n_k$ can achieve for each $k$. Unfortunately,
a naive search is not feasible for determining this value for even
modest values of $k$, so heuristics are needed to trim the search
space. We start by noting that we may assume that $n_2 = 11$ since if
$[n_2, n_3, \dots , n_k]$ satisfies \cref{eq:shor1,eq:shor2,eq:shor3},
then $[11, n_3, \dots, n_k]$ also satisfies
\cref{eq:shor1,eq:shor2,eq:shor3}. Unfortunately, this greedy nature
does not generalise to the remaining parts of the sequence. In order
to minimise $n_k$, we may need to use non-optimal values for
$n_3,\dots,n_{k-1}$. For example, $n_4 = 28$ is attainable. However,
to achieve $n_5=41$, we must use $n_4=31$.

In computations to find the smallest value that $n_k$ can take, the
following heuristic proved very useful. Suppose that we have a current
candidate for the smallest value of $n_k$, say $\kappa$. We say that a
pair $(x,y)$ is \textit{viable} if there exists a sequence
$[\n_x=y,\n_{x+1},\dots,\n_k]$ that satisfies
\cref{eq:shor1,eq:shor2,eq:shor3} and $\n_k < \kappa$. Note that
\cref{eq:shor3} need only be satisfied for $x < j < k$ and that we may
assume that $(x+1,\n_{x+1}),\ldots,(k-1,\n_{k-1})$ are viable. Thus,
in any further computation, we should only use viable pairs 
$(x,y)$.

To determine the smallest possible value for $n_k$, we constructed
sequences that satisfied \cref{eq:shor1,eq:shor2,eq:shor3} and only
contained viable pairs. Each time we found a sequence that had a smaller
value for $n_k$, we recomputed which pairs of $(i,j)$ were viable. Note
that a pair that was marked as not viable could never be marked as
viable for a smaller value of $n_k$, so it never needs to be reconsidered.
If for some $x$ there were no viable pairs $(x,y)$ then we concluded that
the current candidate for the smallest value of $n_k$ is indeed the
smallest value possible. At that point we could begin the search for the
smallest possible value of $n_{k+1}$.

Formalising this approach, say that $v(x,y)$ is true if the pair
$(x,y)$ is viable and false otherwise. Let
$s_k=[\n_{k,2},\n_{k,3},\dots,\n_{k,k}]$ denote our current candidate for a
sequence achieving the lowest possible value for $n_k=\n_{k,k}$.
We start
with $s_2=[11]$.  Then for $k=3,4,\dots,$ we seek the sequence $s_k$
as follows:
\begin{enumerate}
\item Let $s_k$ be the sequence formed from $s_{k-1}$ by appending the
  smallest value that ensures that $s_k$
  satisfies \cref{eq:shor1,eq:shor2,eq:shor3} when we put
  $n_i=\n_{k,i}$ for each $i$.
\item For each $j$, let $v(2,j)$ be true if and only if $j=11$
  and let $v(k,j)$ be true if and only if $\n_{k-1,k-1}+2k\le j<\n_{k,k}$.
\item For $i=k-1,k-2,\dots,3$; for each $j$, let
  $v(i,j)$ be true if and only if $j\ge \n_{i,i}$ and $v(i+1,j')$ is true
  for some $j'\ge j+2i+2$.
\item For $i=k-1,k-2,\dots,2$; for each $j$ for which $v(i,j)$ is true,
  check that there is a sequence $[t_i=j,t_{i+1},\dots,t_k]$ such that
  putting $n_x=t_x$ for $i\le x\le k$ satisfies
  \cref{eq:shor1,eq:shor2,eq:shor3}, and
  where $v(x,t_x)$ is true for $i\le x\le k$.
  If not, set $v(i,j)$ to false.
\item If there is some $i$ for which $v(i,j)$ is false for all $j$, then
  declare the present $s_k$ to be the best possible, and
  begin looking for the best $s_{k+1}$.  Otherwise, $v(2,11)$ must be
  true. Replace $s_k$ by the sequence $[t_2,\dots,t_k]$ that showed that
  $v(2,11)$ is true. Mark $v(k,j)$ to be false for $j\ge t_k$. Go to step 4.
\end{enumerate}

Using the algorithm outlined above we were able to compute the
smallest values that $n_k$ can take and satisfy \cref{lem:shor-eq},
for $k\le21$. The results are presented in \Cref{tab:n_k}. Note that
the same algorithm is capable of computing several further values
for $n_k$, with each additional value requiring a few days of
computation.

\begin{table}[htp]\centering
  \begin{tabular}{lll}
    \toprule
    $k$ & One sequence $[n_2, \dots, n_k]$ that minimises $n_k$ \\
    \midrule
    2  & $[11]$ \\
    3  & $[11,17]$ \\
    4  & $[11,17,28]$ \\
    5  & $[11,17,31,41]$ \\
    6  & $[11,17,28,46,58]$ \\
    7  & $[11,17,28,42,64,78]$ \\
    8  & $[11,17,28,42,63,90,107]$ \\
    9  & $[11,17,28,46,58,91,122,140]$ \\
    10 & $[11,17,28,42,64,78,122,157,177]$ \\
    11 & $[11,17,28,42,63,90,107,165,204,226]$ \\
    12 & $[11,17,28,46,58,91,122,140,216,259,283]$ \\
    13 & $[11,17,28,42,64,78,122,157,177,272,320,346]$ \\
    14 & $[11,17,28,42,64,78,122,157,177,272,356,408,436]$ \\
    15 & $[11,17,28,42,63,90,107,165,204,226,346,439,495,525]$ \\
    16 & $[11,17,28,46,58,91,122,140,216,259,283,432,534,594,626]$ \\
    17 & $[11,17,28,42,64,78,122,157,177,272,320,346,527,638,702,736]$ \\
    18 & $[11,17,28,42,64,78,122,157,177,272,356,408,436,662,783,851,887]$ \\
    19 & $[11,17,28,42,63,90,107,165,204,226,346,439,495,525,796,933,1005,1043]$ \\
    20 & $[11,17,28,46,58,91,122,140,216,259,283,432,534,594,626,948,1110,1192,1234]$ \\
    21 & $[11,17,28,42,64,78,122,157,177,272,320,346,527,638,702,736,1114,1304,1400,1449]$ \\
    \bottomrule
  \end{tabular}
\caption{\label{tab:n_k}Smallest values of $n_k$ that satisfy
  \cref{lem:shor-eq} for $2 \leq k \leq 21$ and one possible sequence
  of $[n_2,\dots,n_k]$ achieving the claimed value.}
\end{table}

\Cref{tab:n_k} can be used to show explicit bounds on the length of a
partial transversal in a Latin square. For example, \cref{tab:n_k}
shows that $n_{21} \geq 1449$. Thus, any Latin array of order
$n<1449$ has a partial transversal of length at least $n-20$.
By comparison, \tref{thm:n-sqrt-n} only implies that any Latin array of order
$n<21^2=441$ has a partial transversal of length at least $n-20$.

  \let\oldthebibliography=\thebibliography
  \let\endoldthebibliography=\endthebibliography
  \renewenvironment{thebibliography}[1]{%
    \begin{oldthebibliography}{#1}%
      \setlength{\parskip}{0.25ex}%
      \setlength{\itemsep}{0.25ex}%
  }%
  {%
    \end{oldthebibliography}%
  }

\end{document}